\documentclass[11pt]{amsart}

\usepackage{graphicx,color,latexsym}
\usepackage{amssymb}
\usepackage{pb-diagram}
\usepackage{verbatim}

\newtheorem{theorem}{Theorem}[section]

\newtheorem{lemma}[theorem]{Lemma}
\newtheorem{cor}[theorem]{Corollary}

\theoremstyle{definition}
\newtheorem{remark}[theorem]{Remark}

\newtheorem{definition}[theorem]{Definition}
\newtheorem{conjecture}[theorem]{Conjecture}

\newtheorem{example}[theorem]{Example}

\begin{document}

\title{Kodaira Dimension of Fiber Sums along Spheres}

\author{Josef G. Dorfmeister}
%\address{School  of Mathematics\\  University of Minnesota\\ Minneapolis, MN 55455}
\address{Institut f\"ur Differentialgeometrie\\  Leibniz Universit\"at Hannover\\ 30167 Hannover, Germany}
\email{j.dorfmeister@math.uni-hannover.de}

\date{\today}
\begin{abstract}In this note we discuss the effect of the symplectic sum along spheres in symplectic four-manifolds on the Kodaira dimension of the underlying symplectic manifold.  We find that the Kodaira dimension is non-decreasing.  Moreover, we are able to obtain precise results on the structure of the manifold obtained from the blow down of an embedded symplectic $-4$-sphere. 
%Further, we show that a symplectic sum along a sphere, as in the higher genus case, does not produce any new diffeomorphism classes of symplectic manifolds with Kodaira dimension 0.  

\end{abstract}

\maketitle

%\tableofcontents

\section{Introduction}

Minimality of symplectic fiber sums in four-manifolds has been researched in \cite{U} and \cite{D} and precise criteria under which such a sum is non-minimal have been found.  Symplectic Kodaira dimension is defined on the minimal model of a symplectic manifold and, given our understanding of the behavior of symplectic sums, it is of interest what the Kodaira dimension of a given symplectic sum is.

This question has been researched for fiber sums along submanifolds of genus strictly greater than 0, see \cite{LY} and \cite{U2}.  In this note we complete the discussion for fiber sums along spheres.  Throughout we work on four-manifolds.  

The symplectic fiber sum is a surgery on two symplectic manifolds $X$ and $Y$, each containing a copy of a symplectic submanifold $V$.  The sum $X\#_VY$ is again a symplectic manifold.  Section \ref{pre} provides a brief overview of the symplectic sum construction and minimality of symplectic manifolds.  A minimal symplectic manifold contains  no exceptional spheres, i.e. no embedded symplectic spheres of self-intersection $-1$.  Furthermore we review the definition of the symplectic Kodaira dimension $\kappa(X)$ for a symplectic manifold $(X,\omega)$   and state some relevant results.

Section \ref{kodsection} is devoted to the proof of the main result, the genus 0 case of the following theorem (the higher genus proof can be found in \cite{LY}):

\begin{theorem}Let $M=X\#_VY$ be a symplectic fiber sum along an embedded symplectic surface $V$ in the four-manifolds $X$ and $Y$.  Then the symplectic Kodaira dimension is non-decreasing, i.e.
\[
\kappa(M)\ge \max\{\kappa(X),\kappa(Y),\kappa(V)\}.
\]
\end{theorem}
This theorem in genus 0 is reduced to a much simpler statement by a result of McDuff (Thm 1.4, \cite{M}):
\begin{theorem}
Let $M=X\#_VY$ be a symplectic fiber sum on four-manifolds along a symplectic hypersurface $V$ of genus 0 and $\kappa(X)\ge 0$.  Assume there exist no symplectic exceptional spheres disjoint from $V$ in $X$ or $Y$.    Then 
\[
\kappa(M)\ge \kappa(X).
\]

\end{theorem}
The proof of this theorem divides nicely into two cases:  $[V]^2\ne \pm 4$ and the $-4$-blow-down of an embedded symplectic $-4$-sphere.  In the former case, the result follows almost immediately.  The latter is proven in two steps:  If there exist disjoint smoothly embedded exceptional spheres transversely intersecting the hypersurface $V_X$ in a single positive point, then a result of Gompf (Lemma \ref{flip}) provides the key to the proof.  If no such spheres exist, then the proof relies on the fact that we can exclude the existence of any symplectic exceptional spheres which meet the hypersurface $V_X$ negatively.

As an auxiliary result, we obtain precise statements on the structure of the $-4$-blow-down $M$ and the manifolds $X$ admitting such a blow-down.  In particular, we find that if $\kappa(X)\ge 0$, then the only relatively minimal symplectic manifolds $X=X_m\#k\overline{\mathbb CP^2}$ containing embedded symplectic $-4$-spheres have $k\le 4$.  And if $X$ contains such a sphere with $k>0$, then $M=X_m\#(k-1)\overline{\mathbb CP^2}$, where $X_m$ is the minimal model of $X$.

Furthermore, the results lead to the following upper bound on the Kodaira dimension:

\begin{lemma}Let $M$ be the $-4$-blow-down of a symplectic $-4$-sphere $V_X$ for a relatively minimal pair $(X^4,V_X)$ with $\kappa(X)\ge 0$.  Then 
\[
\kappa(M)\le \kappa(X)+1.
\]

\end{lemma}
This result needs only to be proven for $\kappa(X)=0$, see Lemma \ref{k=0}.  Actually, this result also holds for irrationally ruled manifolds, see Cor. \ref{irrat}.  Furthermore, the calculations in Section \ref{rat} lead to the following conjecture:
\begin{conjecture}\label{conj}
Let $M$ be the $-4$-blow-down of an embedded symplectic $-4$-sphere $V_X$ for a relatively minimal pair $(X,V_X)$ with $\kappa(X)=-\infty$.  Then 
\[
\kappa(M)\le 0.
\]
\end{conjecture}
If $X$ is rational, i.e. $X=\mathbb CP^2\#k\overline{\mathbb CP^2}$ ($k\ge 0$) or $S^2\times S^2$, then it is shown in Section \ref{rat} that $\kappa(M)<2$.  

Finally, we address the question of which manifolds $X$ can be rationally blown-down to produce a manifold with $\kappa(M)=0$.  In \cite{U2}, it was shown that fiber sums along symplectic submanifolds with positive genus do not produce any new diffeomorphism classes of symplectic manifolds having symplectic Kodaira dimension 0.  In Section \ref{kod0} we show that this also holds true in the genus 0 case.  More precisely, we show that the only minimal manifold $M$ with $\kappa(M)=0$ which can be obtained from a fiber sum along a sphere is the Enriques surface.  In the non-minimal case, we find that  $M$ is diffeomorphic to  $Y\#\overline{\mathbb CP^2}$, with $Y$ a K3 surface or an Enriques surface or an as yet unknown surface with $\kappa(Y)=0$,  or  $X_m\#3\overline{\mathbb CP^2}$ with $X_m$ any minimal manifold with $\kappa(X_m)=0$.

{\bf Acknowledgments}  I gratefully acknowledge the patient support of my advisor Prof. Tian Jun Li.  I also would like to thank Prof. Bob Gompf for pointing out Lemma \ref{nge2} and describing the underlying method.  I also thank Weiyi Zhang and Weiwei Wu for their interest.  Part of this work was completed while at the University of Minnesota.  The author is supported through the Graduiertenkolleg 1463: Analysis, Geometry and String Theory.  

\section{Preliminaries\label{pre}}

This paper deals with the change in Kodaira dimension under symplectic sums along spheres.  In this section we provide an overview of the symplectic sum construction and the relevant results allowing for the definition of symplectic Kodaira dimension.  

\subsection{Symplectic fiber sum}

The symplectic fiber sum is a smooth surgery operation which is performed in the symplectic category.    

Consider first the smooth surgery.  Let $X_1$, $X_2$ be $2n$-dimensional smooth
manifolds. Suppose we are given
codimension 2 embeddings $j_i:V\rightarrow X_i$ of  a
smooth closed oriented manifold $V$ with normal bundles $N_iV$.
Assume that the Euler classes of the normal bundle of the embedding
of $V$ in $X_i$ satisfy $e(N_1V)+e(N_2V)=0$ and fix a
fiber-orientation reversing bundle isomorphism $\Theta:
N_1V\rightarrow N_2V$.  By canonically identifying the normal
bundles with a tubular neighborhood $\nu_i$ of $j_i(V)$, we obtain
an orientation preserving diffeomorphism $\varphi: \nu_1\backslash
j_1(V)\rightarrow \nu_2\backslash j_2(V)$ by composing $\Theta$ with
the diffeomorphism that turns each punctured fiber inside out.  This
defines a gluing of $X_1$ to $X_2$ along the embeddings of $V$ denoted
$M=X_1\#_{(V,\varphi)}X_2$.  The diffeomorphism type of this manifold is
determined by the embeddings $(j_1,j_2)$ and the map $\Theta$.  

In the symplectic category, Gompf (\cite{G}) and McCarthy-Wolfson (\cite{MW}) proved, that if $X_1$ and $X_2$ admit symplectic forms $\omega_1,\omega_2$ resp.  and if the embeddings $j_i$ are symplectic with respect to these forms, then
the smooth manifold $M=X_1\#_{(V,\varphi)}X_2$ comes with a symplectic form
$\omega$ created from $\omega_1$ and $\omega_2$.

This procedure can be reversed, this is called the symplectic cut, see \cite{Le} for details.  Throughout this paper, with the exception of Lemma \ref{flip}, we suppress the map $\varphi$ in the notation and work in the symplectic category.  Hence we refer only to the fiber sum $M=X_1\#_VX_2$.

\subsection{Notation}

Let $(X,\omega_X)$ be a smooth, closed, symplectic 4-manifold.  We shall generally not distinguish between a symplectic form $\omega$ and a symplectic class $[\omega]$, i.e. a cohomology class which can be represented by a symplectic form.  Both shall be denoted by $\omega$.  Let $V$ be a $2$-dimensional smooth closed manifold such that $X$ contains a copy $V_X$ of $V$ which is symplectic with respect to $\omega_X$.  Such a triple $(X,V_X,\omega_X)$ will be called a symplectic pair and we will often suppress the $\omega_X$ in the notation henceforth.  Denote the homology class of $V_X$ by $[V_X]\in H_2(X)$ and the fiber sum of two symplectic pairs $(X,V_X)$ and $(Y,V_Y)$ along $V\cong V_X\cong V_Y$ by  $M=X\#_VY$.

\subsection{Minimality of 4-Manifolds}

\begin{definition} For the smooth manifold $X$, let ${\mathcal E}$ be the set of cohomology
classes whose Poincar\'e dual are represented by smoothly embedded
spheres of self-intersection $-1$.  Similarly, for the symplectic manifold $(X,\omega)$, let ${\mathcal E}_\omega$ denote the set
\[
{\mathcal E}_{\omega}=\{E\in {\mathcal
E}|\hbox{ $E$ is represented by an embedded $\omega-$symplectic
sphere}\}.
\]
$X$ is said to be smoothly (symplectically)
minimal if ${\mathcal E}=\emptyset$ (${\mathcal E}_\omega=\emptyset$).  
\end{definition}

An embedded sphere of self-intersection $-1$ is called an exceptional sphere, its class an exceptional class.  A basic fact proven using SW theory (\cite{T1},
\cite{LL}, \cite{TJL2}) is:

\begin{lemma}
  ${\mathcal E}_{\omega}$ is
empty if and only if ${\mathcal E}$ is empty. In other words, $(X,
\omega)$ is symplectically minimal if and only if $X$ is smoothly
minimal.
\end{lemma}

From a topological perspective, $X$ is minimal if it is not the connected sum of
another manifold $Y$  with $\overline{\mathbb {CP}^2}$.

\begin{definition} 
The manifold $X_m$ is called a minimal model of $X$ if $X_m$ is  minimal and $X$ is the
connected sum $X_m\#k\overline{\mathbb {CP}^2}$ for some $k>0$.
\end{definition}

A manifold $X$ is called rational if its underlying smooth manifold is either $S^2\times S^2\#k\overline{\mathbb CP^2}$ or $\mathbb CP^2\#k\overline{\mathbb CP^2}$ for some $k\ge 0$.  A manifold $X$ is ruled if the underlying smooth manifold is a connected sum of a $S^2$-bundle over a Riemann surface with $k$ copies of $\overline{\mathbb CP^2}$, $k\ge 0$.  The following two results will be useful:

\begin{lemma}\label{inters}\begin{enumerate}Assume $X$ is not rational or ruled.
\item (Thm. 1.5, \cite{M2}) $X_m$ is unique.  
\item (Cor. 3, \cite{TJL2})
No two distinct symplectic exceptional spheres intersect. Moreover,  ${\mathcal E}=\{\pm E_i\}$ where $E_i$ are the generators of the $\overline{\mathbb CP^2}$.
\end{enumerate}
\end{lemma}

The process for obtaining a minimal manifold from $X$ is called blowing down.  This removes the $-1$-sphere and can be obtained as a fiber sum of $(X,E)$ with $(\mathbb CP^2,H)$, i.e. the blown-down manifold $Y=X\#_{E=H}\mathbb CP^2$.  It can be shown, that after blowing down a finite collection of exceptional curves, a minimal manifold is obtained, see Thm 1.1, \cite{M}.

The symplectic fiber sum involves a submanifold $V_X\subset X$.  It is therefore reasonable to consider minimality with respect to this submanifold:

\begin{definition} 
The pair $(X,V_X)$ is called relatively minimal if there exist no exceptional curves $E$ such that $E\cdot [V]=0$.
\end{definition}

The following result shows that we can blow down $X$ in such a manner that $V$ is preserved and the result is relatively minimal:

\begin{lemma}[Thm 1.1ii, \cite{M}] Every symplectic pair $(X,V_X)$ covers a relatively minimal symplectic pair $(\tilde X,V_X)$ which may be obtained by blowing down a finite set of exceptional curves disjoint from $V$.

\end{lemma}

The minimality of symplectic fiber sums is described by the following Theorem:

\begin{theorem}[\cite{U}, \cite{D}]\label{minimal}
Let $M$ be the symplectic fiber sum $X\#_VY$ of the symplectic manifolds $(X,\omega_X)$ and $(Y,\omega_Y)$ along an embedded symplectic surface $V$ of genus $g\ge 0$.  
\begin{enumerate}
\item The manifold $M$ is not minimal if
\begin{itemize}
\item $X\backslash V_X$ or $Y\backslash V_Y$ contains an embedded symplectic sphere of self-intersection $-1$ or
\item $X\#_VY=Z\#_{V_{\mathbb CP^2}}\mathbb CP^2$ with $V_{\mathbb CP^2}$ an embedded $+4$-sphere in class $[V_{\mathbb CP^2}]=2[H]\in H_2(\mathbb CP^2,\mathbb Z)$ and $Z$ has at least 2 disjoint exceptional spheres $E_i$ each meeting the submanifold $V_Z\subset Z$ positively and transversely in a single point with $[E_i]\cdot [V_X]= 1$.
\end{itemize}
\item If $X\#_VY=Z\#_{V_B}B$ where $B$ is a $S^2$-bundle over a genus $g$ surface and $V_B$ is a section of this bundle then $M$ is minimal if and only if $Z$ is minimal.  
\item In all other cases $M$ is minimal.
\end{enumerate}
\end{theorem}

\subsection{Symplectic Kodaira Dimension}

A class $K_M\in H^2(M,\mathbb Z)$ is called a symplectic canonical class if there exists a symplectic form $\omega$ on $M$ such that for any almost complex structure $J$ tamed by $\omega$,
\[
K_M=K_\omega=-c_1(M,J).
\]
We will often suppress the dependence on the symplectic form $\omega$ as our calculations will be unaffected by the precise choice of $\omega$.

The
symplectic Kodaira dimension of a 2 - manifold is defined as follows:

\begin{definition}\label{2dim}\[
\kappa(M,\omega)=\left\{\begin{array}{cc}
-\infty & \hbox{if $K_{\omega}<0$},\\
0& \hbox{ if $K_{\omega}=0$},\\
1& \hbox{ if $K_{\omega}> 0$}.
\end{array}\right.
\]
\end{definition}

A similar definition can be made for 4-manifolds:  The symplectic Kodaira dimension $\kappa^s(M,\omega)$ is defined by Li \cite{L1} (see also \cite{LB1}, \cite{MS}) to be:

{\definition \label{symp Kod}
For a minimal symplectic $4-$manifold $M$ with symplectic form $\omega$
and
symplectic canonical class $K_{\omega}$ the symplectic Kodaira dimension $\kappa^s(M,\omega)$
 is defined in the following way:

\[
\kappa^s(M,\omega)=\left\{\begin{array}{cc}
-\infty & \hbox{if $K_{\omega}\cdot \omega<0$ or $K_{\omega}\cdot K_{\omega}<0$},\\
0& \hbox{ if $K_{\omega}\cdot \omega=0$ and $K_{\omega}\cdot K_{\omega}=0$},\\
1& \hbox{ if $K_{\omega}\cdot \omega> 0$ and $K_{\omega}\cdot K_{\omega}=0$},\\
2& \hbox{ if $K_{\omega}\cdot \omega>0$ and $K_{\omega}\cdot K_{\omega}>0$}.\\
\end{array}\right.
\]

The symplectic Kodaira dimension of a non-minimal manifold is defined to be
that of any of its minimal models.
}

 If the symplectic manifold carries a complex structure $J$ in addition to the symplectic structure, then we can define two classes:  The first Chern class $c_1(\mathcal K_J)$ of the canonical bundle $\mathcal K_J$ and the symplectic canonical class $K_X$. Please note that $c_1(\mathcal K_J)$ and $K_{\omega}$ may differ: If the manifold $M$ admits a K\"ahler structure and is minimal, then the first Chern class of the canonical bundle $\mathcal K_J$ is given by the canonical class $K_X$ of the K\"ahler form and it is unique up to diffeomorphism (see \cite{W} and \cite{FM}). If $J$ and $\omega$ are not compatible, then $c_1(\mathcal K_J)$ and $K_X$  are not necessarily equal. On a non-K\"ahler symplectic manifold there may be many symplectic canonical classes $K_X$ depending on the choice of symplectic structure $\omega$.  Hence the symplectic Kodaira dimension may depend on the choice of symplectic structure $\omega$. 
 
 Moreover, Kodaira dimension was first introduced for complex manifolds.  Holomorphic Kodaira dimension makes use of the canonical bundle $\mathcal K_J$ as well and, if the underlying smooth manifold is complex and symplectic but not necessarily K\"ahler, we potentially have two differing Kodaira dimensions for this manifold.
 
The following Theorem addresses these issues:

\begin{theorem}\label{2.4}(Thm 2.4, \cite{L1}; \cite{DZ}; \cite{liu}) Let $M$ be a closed oriented smooth four manifold and $\omega$ an orientation compatible symplectic form on $M$.  Let $(M,\omega)$ be minimal. 
\begin{enumerate}
\item The symplectic Kodaira dimension is well defined.
\item $\kappa(M,\omega)$ only depends on the oriented diffeomorphism type of $M$, hence $\kappa^s(M,\omega)=\kappa^s(M)$.
\item If $M$ admits a complex structure $J$, then the symplectic Kodaira dimension and the holomorphic Kodaira dimension agree.
\item $\kappa(M)=-\infty$ if and only if $M$ is rational or ruled.
\end{enumerate}
\end{theorem}

Due to this Theorem, we refer only to Kodaira dimension in the following, dropping the superscript $\kappa^s(M)=\kappa(M)$ as well.

\subsection{Splitting of classes under fiber sums}  Kodaira dimension is defined using $K^2_X$ and $K\cdot\omega$.  It is thus useful, to know how products behave under the symplectic sum.

Let $A\in H_2(M,\mathbb Z)$ be represented by a connected surface $C\subset M$ of genus $g$.  The surface $C$ decomposes under a symplectic cut into surfaces $C_X\subset X$ and $C_Y\subset Y$, each surface not necessarily connected.  We shall denote the class of $C_X$ by $A_X$ and similarly for $C_Y$.    

To begin, we have the following:

\begin{lemma}\label{formulas}Given a fiber sum $M=X\#_VY$ of two symplectic pairs $(X,V_X)$ and $(Y,V_Y)$ and a class $A\in H_2(M,\mathbb Z)$ decomposing into $(A_X,A_Y)\in H_2(X,\mathbb Z)\times H_2(Y,\mathbb Z) $ we have
\begin{equation}
K_M\cdot A=K_X\cdot A_X+A_X \cdot [V_X]+K_Y\cdot A_Y+A_Y\cdot [V_Y]=
\end{equation}
and
\begin{equation}
A^2=A_X^2+A_Y^2
\end{equation}
where $A^2=A\cdot A$.

\end{lemma}

\begin{proof}
The first is  Lemma 2.3, \cite{IP5}.  The second is in \cite{Liun}, a proof can be found in \cite{DL3}.
\end{proof}

Note that this Lemma holds only for classes in $H_2(M,\mathbb Z)$.  It will be useful to have a similar formula for symplectic classes.

\begin{lemma}\label{symp}Let $\omega$ be a symplectic class on $M$.  Suppose that $\omega_X$ and $\omega_Y$ are symplectic classes such that symplectic forms representing each sum to a symplectic representative of $\omega$ on $M$.  Then
\[
K_M\cdot \omega=K_X\cdot\omega_X+K_Y\cdot\omega_Y+\int_{V_X}\omega_X+\int_{V_Y}\omega_Y.
\]
\end{lemma}

\begin{proof}
This result follows from the constructions in Section 2, \cite{IP5}; see the proof of Thm. 3.1, \cite{U2}, for details.
\end{proof}

It will be convenient to have results for the behavior of $K^2$ and $K\cdot\omega$ under blow-ups, these results are surely known, see \cite{G}:

\begin{lemma}\label{bd}
Let $Y=X\#_{E=H}\mathbb CP^2$ be the blow-down of the exceptional curve $E$ in $X$.  Let $\omega_Y$ denote the symplectic form produced in the fiber sum from $\omega_X$ and $\omega_{\mathbb CP^2}$.  Then
\[
K_Y\cdot\omega_Y=K_X\cdot\omega_X-\omega_X\cdot E
\]
and
\[
K_Y^2=K_X^2+1.
\]

\end{lemma}

\begin{proof}
Both results can be obtained by direct calculation from Lemmas \ref{formulas} and \ref{symp}.
\end{proof}

\begin{cor}\label{mbd}
Let $Y$ be the successive blow-down of a finite set of $n$ exceptional curves $\{E_i\}$ in $X$.  Assume that $X$ is not rational or ruled.  Let $\omega_Y$ be the symplectic form produced from $\omega_X$ under the blow down.  Then
\[
K_Y\cdot\omega_Y=K_X\cdot\omega_X-\sum_{i=1}^n\omega_X\cdot E_i
\]
and 
\[
K_Y^2=K_X^2+n.
\]
\end{cor}

\begin{proof}
The assumption on $X$ ensures that no two exceptional curves intersect.  Hence the blow-down of $E_i$ does not change $E_j$ for any $i\ne j$.  Moreover, the symplectic form obtained on the blow-down agrees with the original symplectic form away from the gluing locus, hence the symplectic area of $E_j$ is also unchanged.  The result now follows from repeatedly applying Lemma \ref{bd}.
\end{proof}

This result also holds in the case $X$ is rational or ruled if we assume that the set of exceptional divisors $\{E_i\}$ consists of pairwise disjoint curves.

\section{Kodaira dimension\label{kodsection}}

This section is devoted to the Kodaira dimension of symplectic sums.  The change in symplectic Kodaira dimension of a symplectic manifold under the fiber sum operation along surfaces of positive genus has been studied in \cite{LY} and \cite{U2}.  It was shown that if $M=X\#_VY$, then $\kappa(M)\ge \max\{\kappa(X),\kappa(Y),\kappa(V)\}.$  In this section we would like to show, that the same holds for fiber sums along surfaces of genus 0:  

\begin{theorem}\label{main1}Let $M=X\#_VY$ be a symplectic fiber sum along a symplectic hypersurface $V$ of genus 0.  Then the symplectic Kodaira dimension is non-decreasing, i.e.
\[
\kappa(M)\ge \max\{\kappa(X),\kappa(Y),\kappa(V)\}.
\]

\end{theorem}

We first restate the Theorem in a simpler fashion.  As we have seen, for the symplectic sum to exist, we must have $[V]^2\ge 0$ in either $X$ or $Y$.  We will assume that this holds in $Y$.  As described in \cite{D}, the symplectic sum along a sphere thus reduces to one of the following four cases for $(Y,V_Y)$ (Thm. 1.4, \cite{M}, \cite{G}):
\begin{itemize}
\item $(Y,V_Y)=(\mathbb CP^2,H)$, 
\item $(Y,V_Y)=(\mathbb CP^2,2H)$,
\item $Y$ an $S^2$-bundle over a genus $g$ surface, $V_Y$ a fiber, or
\item $Y$ an $S^2$-bundle over sphere, $V_Y$ a section.
\end{itemize} 
Hence  $\kappa(Y)=\kappa(V)=-\infty$ and thus the statement of Thm. \ref{main} is equivalent to
\[
\kappa(M)\ge \kappa(X).
\]  
Moreover, if $\kappa(X)=-\infty$ the result is trivially true.  We therefore assume in the following that $\kappa(X)\ge 0$.  Note that this in particular ensures that any non-minimal $X$ has a unique minimal model $X_m$.

Theorem \ref{minimal} shows, if $X$ is not relatively minimal then we obtain an exceptional curve in $M$ from each exceptional curve not meeting $V_X$ in $X$.  Blowing down both sets of curves does not change the Kodaira dimension and the blown down  manifold $X'$ can be summed with $Y$ along $V_X$ to obtain the blown down manifold $M'$.  We may thus assume, that all of our calculations are done on a relatively minimal pair. 

\begin{theorem}\label{main}Let $M=X\#_VY$ be a symplectic fiber sum along a relatively minimal (in $X$ and $Y$) symplectic hypersurface $V$ of genus 0 and $\kappa(X)\ge 0$.  Then 
\[
\kappa(M)\ge \kappa(X).
\]

\end{theorem}

%
%Let $n_{sm}$ denote the number of disjoint smooth exceptional curves $E$ meeting the hypersurface $V$ in a single point locally positively and tranverselly of order 1 and let $n$ denote the number of symplectic exceptional curves with the same behavior.  Clearly $n_{sm}\ge n$.  As it turns out, we will see that $n=n_{sm}$ if $\kappa(X)\ge 0$.
%
%When $n_{sm}>0$, it follows from a result of Gompf that the smooth fiber sum is diffeomorphic to the blow-down of an exceptional sphere in $X$, see Lemmas \ref{n=1} and \ref{nge2}.  From this Theorem \ref{main} follows immediately.  If $n_{sm}=0$, we will need to explicitly calculate the change of the value of $K^2$ and $K\cdot\omega$ under the symplectic sum.  Moreover, these results allow us to explicitly determine the diffeomorphism class of the symplectic sum $M=X\#_VY$ for certain values of $\kappa(X)$ and $n$.

\subsection{ $\bf [V]^2\ne\pm4$}

We now prove Thm. \ref{main} for hypersurfaces with $[V]^2\ne\pm4$.  From the list above, either $Y$ is an $S^2$-bundle or $\mathbb CP^2$.  In the latter case this is just the classical blow-down of an exceptional sphere in $X$.

\begin{lemma}
If $(Y,V_Y)\ne(\mathbb CP^2,2H)$ then $\kappa(M)=\kappa(X)$.
\end{lemma}

\begin{proof}
In the blow down case, the Kodaira dimension does not change by definition.  In the case of a $S^2$-bundle, we either do not change the diffeomorphism type ($V_Y$ a section) or obtain again a $S^2$-bundle $(V_Y$ a fiber).  Hence again the Kodaira dimension is unchanged.
\end{proof}

%We assume $\kappa(X)\ge 0$, hence by Lemma \ref{inters} and Thm. \ref{2.4} the number of smooth exceptional curves is finite and thus also the number of symplectic exceptional curves. In the following, denote by $k$ the number of disjoint symplectic exceptional curves in $X$.  Note that there may be exceptional curves $E_i$ with $E_i\cdot[V_X]\ne 1$, i.e. the number of exceptional curves $k$ in $X$ satisfies $k\ge n$.  

\subsection{A Smooth Result}We may restrict ourselves to the case $(Y,V_Y)=(\mathbb CP^2,2H)$ and assume that $(X,V_X)$ is relatively minimal.  The following Lemma of Gompf describes the flipping of an exceptional curve from one side of the smooth sum to the other and its effect on the diffeomorphism type of the sum:  

\begin{lemma}[Lemma 5.1,\cite{G}]\label{flip}
Let $V_X$ and $V_Y$ be  closed, connected, orientable diffeomorphic surfaces in the oriented $4-$manifolds $X$ and $Y$ such that $[V_X]^2+[V_Y]^2=1$.  Let $\tilde V_{\tilde X}$ denote the blow up of a point on $V_X$ and $\tilde X=X\#\overline{\mathbb CP^2}$.  Similarly for $(\tilde Y,\tilde V_{\tilde Y})$.  Then $\tilde X\#_{\tilde V_{\tilde X}=V_Y}Y$ and $X\#_{V_X=\tilde V_{\tilde Y}}\tilde Y$ are diffeomorphic.

\end{lemma}

\begin{remark}
If the symplectic sum is performed in the symplectic category, then results in \cite{MSy} show that $(\tilde X\#_{\tilde V_{\tilde X}=V_Y}Y,\omega_1)$ and $(X\#_{V_X=\tilde V_{\tilde Y}}\tilde Y,\omega_2)$ are weakly deformation equivalent as symplectic manifolds.  This means that there exists a diffeomorphism $\phi:\tilde X\#_{\tilde V_{\tilde X}=V_Y}Y\rightarrow X\#_{V_X=\tilde V_{\tilde Y}}\tilde Y$ such that $\phi^*(\omega_2)$ can be connected to $\omega_1$ by a smooth family of symplectic forms $\Omega_t$, $t\in[0,1]$ such that $\Omega_0=   \phi^*(\omega_2)$ and $\Omega_1=\omega_1$.  Hence in this setting we not only obtain no new smooth structure, we also obtain no exotic symplectic structure.  
\end{remark}

Lemma \ref{flip} is the essential ingredient in the following two results:

\begin{lemma}\label{n=1}[\cite{G}]
Let $(X,V_X)$ be a relatively minimal smooth pair with $V_X$ an embedded $-4$-sphere.  If $X$ contains a smoothly embedded exceptional sphere transversely intersecting the hypersurface $V_X$ in a single positive point, then the manifold obtained under $-4$-blow-down of $V_X$ is diffeomorphic to the blow-down of $X$ along this sphere.
\end{lemma}

The method of proof is similar to the proof of the following Lemma and will hence be omitted.  The following was observed by R. Gompf \cite{Go}:

\begin{lemma}\label{nge2}
Let $(X,V_X)$ be a relatively minimal smooth pair with $V_X$ an embedded $-4$-sphere. If $X$ contains two disjoint smoothly embedded exceptional spheres each transversely intersecting the hypersurface $V_X$ in a single positive point, then the manifold obtained under $-4$-blow-down of $V_X$ is diffeomorphic to the blow-down of $X$ along one of these spheres.  
  
\end{lemma}

\begin{proof}
Lemma \ref{flip} allows us to exchange the two exceptional spheres in $X$ for two exceptional spheres in $\mathbb CP^2$ while producing diffeomorphic sums.  This means we blow-down $X$ twice along the exceptional spheres intersecting $V_X$ in a single point and obtain the pair $(\tilde X,\tilde V_X)$, where $\tilde V_X$ is an embedded symplectic $-2$-sphere.  At the same time, we blow up the $4$-sphere $2H\subset \mathbb CP^2$ in two distinct points to obtain a $2$-sphere intersected twice by the exceptional curves.  Then Lemma \ref{flip} implies that $M=X\#_V\mathbb CP^2$ is diffeomorphic to $\tilde X\#_{\tilde V_X}\left(\mathbb CP^2\#2\overline{\mathbb CP^2}\right)$.

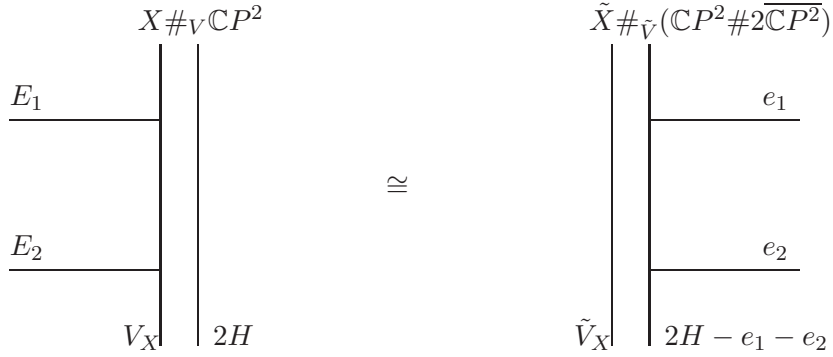
\setlength{\unitlength}{1cm}
\begin{picture}(4,6)(-2,-2)
\put(1.5,3){\line(0,-1){4}}
\put(1,3){\line(0,-1){4}}
\put(1,0){\line(-1,0){2}}
\put(1,2){\line(-1,0){2}}
%\qbezier(1,1)(1.5,0.3)(2,0)
%\qbezier(2,2)(1.5,1.8)(1,1)
%\qbezier(2,2)(3,1)(2,0)

%\put(2,1){\oval(2,2)[r]}
%\put(1.2,1){$\#_V$}
\put(0.7,3.2){$ X\#_V\mathbb CP^2$}
%\put(4,3){$\mathbb CP^2$}
\put(-1,2.2){$E_1$}
\put(-1,0.2){$E_2$}
%\put(1.2,1){$(1,1)$}
\put(1.7,-1){$2H$}
\put(0.5,-1){$V_X$}
\put(4,1){$\cong$}

\put(7.5,3){\line(0,-1){4}}
\put(7,3){\line(0,-1){4}}
\put(7.5,2){\line(3,0){2}}
\put(7.5,0){\line(3,0){2}}
%\put(7,1){\oval(3,2)[r]}

\put(6.7,3.2){$ \tilde X\#_{\tilde V}(\mathbb CP^2\#2\overline{\mathbb CP^2})$}

%\put(5,3){$ X$}
%\put(8,3){$\mathbb CP^2$}
\put(9,2.2){$e_1$}
\put(9,0.2){$e_2$}
%\put(6.5,1.6){$(1)$}
%\put(6.5,0.3){$(1)$}
\put(7.7,-1){$2H-e_1-e_2$}
\put(6.5,-1){$\tilde V_X$}

\end{picture}

The manifold $\mathbb CP^2\#2\overline{\mathbb CP^2}$ now contains 3 interesting exceptional curves: The two exceptional curves $e_1$ and $e_2$ as well as an exceptional curve in class $[H]-e_1-e_2$ obtained from the unique line $H$ going through the two distinct points on the $4$-sphere which are blown-up.  The blow up has the effect of separating the line $H$ from the sphere $2H$.  Moreover, we can blow down the exceptional curve $[H]-e_1-e_2$ without changing the $2$-sphere.  After this blow-down, we are left with $S^2\times S^2$ and the $2$-sphere is a section of this bundle.  Note that this is a minimal manifold.  The sum with $(\tilde X,\tilde V_X)$ does not change the diffeomorphism type of $\tilde X$, hence 
\[
M\cong\tilde X\#_{\tilde V_X}\left(\mathbb CP^2\#2\overline{\mathbb CP^2}\right)\cong\tilde X\#\overline{\mathbb CP^2}.
\]
\end{proof}

The results of Gompf seem to indicate, that it is not the behavior of symplectic exceptional curves on which we need to concentrate, but rather the behavior of smooth exceptional curves meeting the symplectic hypersurface $V$.  For this reason we make the following definitions:

\begin{definition}Define the following two sets on the symplectic pair $(X,V)$ for a fixed symplectic form $\omega$:  \begin{itemize}
\item $N_{sm}=\{e\in\mathcal E\,\vert\;$ e is represented by a curve transversely intersecting the hypersurface $V_X$ in a single positive point$ \}$
\item $N_{sy}=\mathcal E_\omega\cap N_{sm}$
\end{itemize}
Denote $n_{sm}=\#N_{sm}$ and  $n_{sy}=\#N_{sy}$.  
\end{definition}

Clearly $n_{sy}\le n_{sm}$.  Moreover, if $\kappa(X)\ge 0$, then by Lemma \ref{inters} we have $n_{sm}<\infty$.  However, if $\kappa(X)=-\infty$, then the finiteness of $N_{sm}$ or even $N_{sy}$ can no longer be guaranteed.

\begin{cor}Let $(X,V_X)$ be a symplectic pair with $V_X$ an embedded $-4$-sphere and $n_{sm}\ge 1$.  Then $\kappa(X)=\kappa(M)$.
\end{cor}

\begin{proof}
The Kodaira dimension depends only on the oriented diffeomorphism type, see Thm. \ref{2.4}.  Thus $M$ has the diffeomorphism type of the blow-down of an exceptional sphere in $X$ and thus has unchanged Kodaira dimension.
\end{proof}

This completely answers the question concerning Kodaira dimension if $n_{sm}\ge 1$.

An immediate result of this is that $-4$-spheres produced by 3-point blow-ups do not lead to a change in Kodaira dimension.   

\begin{definition}
Let $(X,V)$ be a smooth pair with $V$ an embedded smooth $-4$-sphere produced by blowing up three distinct points on an exceptional sphere.  We call such a $-4$-sphere $V$ artificial.
\end{definition}

\begin{remark}
Assume that the artificial sphere $V_X$ is obtained by symplectically blowing up a symplectic manifold $\tilde X$.  Then we have $n_{sm}\ge 4$ but $n_{sy}\ge 3$ and $n_{sy}<n_{sm}$ if both are finite.  Hence the numbers $n_{sy}$ and $n_{sm}$ can differ.
\end{remark}

\begin{lemma}If $(X,V_X)$ is an artificial symplectic pair and $M$ the $-4$-blow-down along $V$, then $\kappa(X)=\kappa(M)$. 
\end{lemma}

For example, if $X$ is irrational ruled, then it was shown in \cite{D}, that every $-4$-sphere  is produced from the 3 point blow-up of an exceptional sphere.  Moreover, there is always a symplectic form making $V_X$ symplectic and admitting  symplectic disjoint embedded exceptional spheres each transversely intersecting the hypersurface $V_X$ in a single positive point.  Thus every symplectic pair $(X,V_X)$ with $V_X$ an embedded symplectic $-4$-sphere in an irrational ruled manifold is artificial.  

\begin{cor}\label{irrat}
 If $X$ is irrational ruled, then $\kappa(X)=\kappa(M)=-\infty$ for any $-4$-blow-down.
\end{cor}

\subsection{The Symplectic Setting for $n_{sm}=0$}

The smooth results in the previous section provide no hint as to the behavior of the $-4$-blow-down in the case that there are no smooth exceptional curves satisfying the restrictions of Lemma \ref{nge2}.  This section deals with the $-4$-blow-down of a relatively minimal symplectic $-4$-sphere $V$ in a non-rational or ruled symplectic manifold $X$ with $n_{sm}=0$.

The key ingredient in the proof of Theorem \ref{main} in this case is the non-existence of symplectic exceptional curves meeting the hypersurface $V_X$ negatively.  The results in \cite{TJL2} ensure that every smooth exceptional sphere is $\mathbb Z$-homologous to a symplectic exceptional sphere, up to sign.  If the existence of exceptional curves with $e\cdot [V]<0$ could be excluded, then to every smooth exceptional curve could be associated a symplectic exceptional curve, i.e. $n_{sm}=n_{sy}$.  Moreover, all our arguments could make use of symplectic methods.

\begin{lemma}\label{k=t}
Let $(X,V_X)$ be a relatively minimal symplectic pair with $\kappa(X)\ge 0$ and $V_X$ an embedded symplectic $-4$-sphere.  Then either $V_X$ is artificial or there exists no symplectic exceptional curve $E\subset X$ such that $E\cdot [V_X]<0$.
\end{lemma}

\begin{proof}
Let $E_i$, $1\le i\le t$, denote exceptional classes with $E_i\cdot [V_X]= m_i > 0$.  Assume that $E$ is an exceptional class with $E\cdot[V_X]<0$.    

For generic $\omega$-compatible $J$, E is represented by an embedded $J$ - holomorphic submanifold.  Hence, for any $J$, we can find a limiting $J$-holomorphic curve.  By positivity of intersections and unique continuation, it follows that $E=B +m[V_X]$ with $B\cdot[V_X]>0$ and $m>0$.  We may assume that $B$ is represented by embedded, possibly disconnected spheres, some of which may be bubbles appearing in the limit.  Smoothing this curve will produce a symplectic embedded submanifold representing $E$, albeit not necessarily a J-holomorphic submanifold, again keeping the artificial case in mind.

The equality $E_i\cdot E=0$ implies that $(B+m[V_X])\cdot E_i=0$, hence $B\cdot E_i =-m[V_X] \cdot E_i=-mm_i<0$.  Hence we can write
\[
B=\sum_{j=1}^{s}B_{j}+\sum_{i=1}^{n_{sy}}mE_i+\sum_{i=n_{sy}+1}^tmm_iE_i\;\;\;\;\;B_{j}\cdot E_i=0\;\;\;\;B_j\cdot[V_X]\ge 0
\]
where each $B_j$ represents the class of a non-exceptional, possibly multiply covered, embedded sphere.

We now return to $(B+m[V_X])\cdot[V_X]=E\cdot [V_X]<0$.  This implies that $0<B\cdot [V_X]< 4m$.  Hence we obtain from the above decomposition that 
\[
0< \sum_{j=1}^{s}B_{j}\cdot[V_X]+ \sum_{i=1}^t mm_iE_i\cdot[V_X]<4m
\]  
which implies that 
\[
\sum_{i=1}^tmm_iE_i\cdot[V_X]=n_{sy}m+\sum_{i=n+1}^tmm_i^2<4m\;\mbox{  and  }\;m_i\ge 2.
\]
Thus we obtain $n_{sy}=t$ under the assumption $E\cdot [V_X]<0$.

As $E$ is an exceptional curve, we have $K\cdot E=-1$ and $E^2=-1$.  The first leads to
\[
-1=K\cdot (\sum_{i=1}^s B_i+m\sum_{i=1}^{n_{sy}} E_i+m[V_X])=\sum_{i=1}^s K\cdot B_i+(2-n_{sy})m
\]
which implies that 
\[
\sum_{i=1}^s K\cdot B_i\le -1.
\]
Each $B_i$ can be written as $B_i=b_i\tilde B_i$ where $\tilde B_i$ represents the embedded sphere and $b_i$ denotes the degree of the covering.  As $\kappa(X)\ge 0$, we cannot have $K\cdot \tilde B_i\le -2$ for any $i$, hence we must have $K\cdot \tilde B_i=-1$ for at least one $i$.  All classes with $K\cdot \tilde B_i\ge 0$ must be spheres with self-intersection $\le -2$, hence we can prevent the appearance of such a sphere through a generic choice of $J$ making $V_X$ $J$-holomorphic.  It follows that we must have 
\[
s=1,\;B_1=\tilde B_1\mbox{ and }\; K\cdot B_1=-1.
\]
Therefore $B_1$ is an exceptional sphere, which we have disallowed.
%Note also that the curve representing $E$ has genus 0, hence $B_1\cdot[V_X]=1$.  Otherwise we would obtain a curve with genus $>0$ representing E after smoothing.  
%
%The previous results allow us to write 
%\[
%E=B_1+m\sum_{i=1}^{n_{sy}} E_i+m[V_X].
%\]
%
%Now we make use of $E^2=-1$:  This leads to
%\[
%B_1^2=m^2(4-n_{sy})-2m-1.
%\]
%
%The assumption $\kappa(X)\ge 0$ implies that the sphere representing $B_1$ must satisfy $B_1^2<0$ which, together with $0\le n\le 2$,  means that we must have $m=1$.  Hence $B_1^2=1-n$.  It immediately follows that for $n=0,1$ we cannot have $B_1^2<0$ and if $n=2$, then $B_1$ is an exceptional sphere, which we have disallowed.  

The class $E$ can therefore be written as
\[
E=m\left(\sum_{i=1}^{n_{sy}}E_i+[V_X]\right)
\]
with 
\[
-1=E^2=m^2(-n_{sy}+2n_{sy}-4)=m^2(n_{sy}-4).
\]
Thus $m=1$ and $n_{sy}=3$.  Hence we cannot have an exceptional curve meeting $V_X$ negatively unless $V_X$ is artificial.

\end{proof}

\begin{cor}\label{n=n}
Let $(X,V_X)$ be a relatively minimal symplectic pair with $\kappa(X)\ge 0$ and $V_X$ an embedded symplectic $-4$-sphere which is not artificial.  Then $n_{sm}=n_{sy}$.
\end{cor}

\begin{proof}
By Cor. 3, \cite{TJL2}, every smooth exceptional curve is $\mathbb Z$-homologous to a symplectic exceptional sphere, up to sign.  If $e$ is represented by a smooth exceptional sphere and $e\in N_{sm}$, then $e\cdot [V_X]=1$.  Hence $-e$ cannot be represented by a symplectic sphere as $-e\cdot[V_X]=-1<0$ contradicting Lemma \ref{k=t}.  Thus, by Lemma \ref{inters}, we have $N_{sm}=N_{sy}$ and hence $n_{sm}=n_{sy}$.
\end{proof}

\begin{lemma}\label{n=t}
Let $(X,\omega_X)$ contain a symplectic $-4$-sphere $V_X$ and assume $(X,V_X)$ is relatively minimal non-artificial pair with $\kappa(X)\ge 0$.  Assume that $X=X_m\#k\overline{\mathbb CP^2}$.  Then $n_{sm}<4$.  Moreover, $X=X_m\#n_{sm}\overline{\mathbb CP^2}$ if $n_{sm}>0$ and $X=X_m$ or $X_m\#\overline{\mathbb CP^2}$ if $n_{sm}=0$.  
\end{lemma}

\begin{proof}
We have already shown that $n_{sm}=n_{sy}$ if $\kappa(X)\ge 0$ and $V_X$ is non-artificial.  Thus we can dispense with the subscript in the following.

{$\bf b^+>1$:}   Then for fixed $(X,V_X,\omega)$, the canonical class $K_X=K$ is represented by an embedded, possibly disconnected, symplectic surface.  Moreover, for generic almost complex structures compatible with $\omega$, this submanifold is $J$-holomorphic.  Thus for any $J$ making $V_X$ $J$-holomorphic, we can find a $J$-holomorphic representative of $K$ via Gromov convergence.  The limit curve $C$ may now have components with bubbles or lying in $V_X$.  We take this into account in the following.  

The canonical class decomposes as follows, each component representing a $J$-holomorphic curve:
\begin{itemize}
\item $A\in H_2(X)$ with $A\cdot [V_X]=0$,
\item $B_0\in H_2(X)$, not an exceptional class, with $B_0\cdot [V_X]>0$,
\item $B\in H_2(X)$, not an exceptional class, with $B\cdot V_X<0$ and
\item $E_i$ exceptional classes with $E_i\cdot [V_X]= m_i > 0$
\end{itemize}
such that 
\[
K=A+B_0+B+\sum_{i=1}^kE_i.
\]
By positivity of intersections and unique continuation, it follows that $B=\sum_i(B_i +b_i[V_X])$ with each $(B_i+b_i[V_X])\cdot[V_X]<0$ while $B_i\cdot[V_X]>0$.  

Note, if $V_X$ itself is a connected component of the limit curve $C$ representing $K$, then $X$ is minimal.  This follows by contradiction, assuming $n>0$ and using positivity of intersections and $K\cdot E=-1$.  Thus in this case $n=k=0$.  

Assume that $B=0$.  Then
\[
2=K\cdot [V_X]=B_0\cdot[V_X]+\sum_{i=1}^kE_i\cdot[V_X]=B_0\cdot[V_X]+n+\sum_{i=n+1}^kE_i\cdot[V_X]
\]
where $E_i\cdot[V_X]\ge 2$.  Hence $n\le 2$ and if $n=1,2$ we obtain $k=n$ while if $n=0$ we could have at most $k=1$.

Assume now that $B\ne 0$.  We proceed as in the proof of Lemma \ref{k=t}.  The equality $K\cdot E_j=-1$ implies that $(B_i+b_i[V_X])\cdot E_j=0$, hence $B_i\cdot E_j =-b_i[V_X] \cdot E_j=-b_im_j$.  Hence we can write
\[
B_i=\sum_{j=1}^{s_i}B_{ij}+\sum_{j=1}^kb_im_jE_j\;\;\;\;\;B_{ij}\cdot E_k=0\;\;\;\;\;B_{ij}\cdot [V_X]\ge 0.
\]

We now return to $(B_i+b_i[V_X])\cdot[V_X]<0$, this implies that $0< B_i\cdot [V_X]< 4b_i$.  Hence we obtain from the above decomposition that 
\[
0< \sum_{j=1}^{s_i}B_{ij}\cdot[V_X]+ \sum_{j=1}^k b_im_jE_j\cdot[V_X]<4b_i
\]  
which implies that 
\[
s_i+b_in+\sum_{j=n+1}^kb_im_j^2<4b_i,\;\;m_j\ge 2.
\]
Thus we obtain $n=k<4$.

{$\bf b^+=1$:} The proof in the $b^+>1$ case fails to transfer as we cannot ensure that the canonical class is represented by an embedded symplectic submanifold.  However, under the assumption $\kappa(X)\ge 0$, Prop. 5.2 in \cite{LL2}, ensures that $A=\pi^*(qK_{X_m})+\sum_iE_i$ is representable by an embedded symplectic submanifold $C_{q}$ for $q\ge 2$.  Note that this does not depend on the choice of almost complex structure $J$ compatible with $\omega$.  As before, we can take a limit curve and decompose $A$ as in the $b^+>1$ case.  In particular, if $B\ne 0$, then the argument is exactly as before as we still have $A\cdot E_i=-1$.  If we have $B=0$, then  $A\cdot [V_X]\ge 0$.  Using $2q-(q-1)\sum_i E_i[V_X]=A\cdot[V_X]$ it follows that
\[
3>\frac{2q}{q-1}\ge n+\sum_{i=n+1}^kE_i\cdot[V_X]
\]
for $q$ large enough.  Recall that $E_i\cdot[V_X]\ge 2$ for all $E_i$ in the last sum.  Thus again we have $n\le 2$ and  $k=n$ if $n=1,2$ and $k\le 1$ if $n=0$.
\end{proof}

With these results we are now ready to complete the proof of Theorem \ref{main}.  The following simple Lemma notes the change in $K^2$ and $K\cdot \omega$ under $-4$-blow-down:

\begin{lemma}\label{rbd}
Let $(X,\omega_X)$ contain a symplectic $-4$-sphere $V_X$ and denote $(M, \omega_M)$ the $-4$-blow-down of $V_X$.  Assume that $V_X\subset X$ and $2H\subset \mathbb CP^2$ have the same symplectic area.  Then
\[
K_M\cdot\omega_M=K_X\cdot\omega_X+\frac{1}{2}\omega_X\cdot[V_X]
\]
and
\[
K_M^2=K_X^2+1.
\]

\end{lemma}

\begin{proof}
Lemma \ref{symp} shows that
\[
K_M\cdot\omega_{M}= K_X\cdot\omega_X-(3[H])\cdot\omega_{\mathbb CP^2}+ 2\left(\omega_{\mathbb CP^2}\cdot 2[H]\right).
\]
The assumption on the symplectic areas translates as $\omega_X\cdot[V_X]=\omega_{\mathbb CP^2}\cdot2[H]$.  Hence
\[
K_M\cdot\omega_{M}= K_X\cdot\omega_X+\omega_{\mathbb CP^2}\cdot [H]=K_X\cdot\omega_X+\frac{1}{2}\omega_X\cdot[V_X].
\]
The second result follows directly:
\[
K_M^2=\left(K_X+[V_X]\right)^2+(-3[H]+2[H])^2=K_X^2+1.
\]
\end{proof}

\begin{lemma}\label{noinf}Let the relatively minimal pair $(X,V_X)$ contain a symplectic $-4$-sphere $V_X$ and let $M$ be the $-4$-blow-down along $V_X$.
Assume that $\kappa(X)\ge 0$ and $n_{sm}=0$.  Then $\kappa(M)\ge 0$.
\end{lemma}

\begin{proof}
As $n_{sm}=n_{sy}=0$, Thm. \ref{minimal} ensures that $M$ is minimal.  Hence we obtain from Lemma \ref{n=t} that 
\[
K_X^2\ge -1
\]
and thus with Lemma \ref{rbd}
\[
K_M^2\ge 0.
\]
%It was shown in \cite{LS}, that if $\kappa(X)\ge 0$, then all relative genus 0 Gromov-Witten invariants for a non-trivial class $A\in H_2(X)$ vanish.  The sum formula for GW-invariants then stipulates, that any contribution to a genus 0 absolute GW-invariant of $M$ must come from classes $A\in H_2(\mathbb CP^2,\mathbb Z)$ of spheres in $\mathbb CP^2$ with $A\cdot[V_Y]=0$.  However, all $A=a[H]\in H_2(\mathbb CP^2,\mathbb Z)$, $a\ne 0$, intersect $[V_Y]=2[H]$, hence $M$ has no non-vanishing genus 0 GW-invariants for a nontrivial class.  This implies that $\kappa(M)\ge 0$.
\end{proof}

\begin{lemma}\label{Kw}
Let the relatively minimal pair $(X,V_X)$ contain a symplectic $-4$-sphere $V_X$ and let $M$ be the $-4$-blow-down along $V_X$.  Assume that $n_{sm}=0$.  If $\kappa(X)\ge 1$, then $\kappa(M)\ge 1$.
\end{lemma}

\begin{proof}We have shown that $\kappa(X)\ge 0$ implies $\kappa(M)\ge 0$.  Assuming $\kappa(X)\ge 0$, we can state:  $\kappa(X) \ge 1$ if and only if $K_X\cdot\omega_X>0$.  Thus we need only show that $K_{M_m}\cdot \omega_{M_m}>0$ holds for the minimal model $M_m$ of $M$.  

Moreover, Cor. \ref{mbd} shows, that  under blow-ups, $K_X\cdot\omega_X$ increases, i.e. if $\kappa(X)\ge 1$, the sign of $K_X\cdot\omega_X$ is unchanged under blow-ups.  

$M$ is minimal and thus  
\[
K_{M_m}\cdot \omega_{M_m}=K_M\cdot\omega_{M}= K_X\cdot\omega_X+\frac{1}{2}\omega_X\cdot[V_X] >0.
\]
\end{proof}

\begin{lemma}
Assume that $n_{sm}=0$ and $X$ contains a symplectically embedded $-4$-sphere $V_X$ such that $(X,V_X)$ is relatively minimal.  If $\kappa(X)=2$, then $\kappa(M)=2$.
\end{lemma}

\begin{proof}$M$ is again minimal and thus by Lemma \ref{rbd} we have
\[
K_X^2>-1
\]
which implies
\[
K_M^2>0
\]
by Lemma \ref{n=t}.
\end{proof}

This completes the proof of Thm. \ref{main}.

\subsection{The Structure of $M$ for $n_{sm}>0$}  The $-4$-blow-down for $n_{sm}>0$ leaves the Kodaira dimension unchanged and the manifold $M$ is diffeomorphic to the blow-down of $X$ along any exceptional sphere in $N_{sm}$.  However, we can be slightly more precise in some situations.

Lemma \ref{k=t} and \ref{n=t} allow us to determine the $-4$-blow-down $M$ of $X$ if $n_{sm}>0$ and $\kappa(X)\ge0$:

\begin{cor}\label{n2blow}
If $n_{sm}>0$ and $\kappa(X)\ge 0$, then $M$ is diffeomorphic to $X_m\#(n_{sm}-1)\overline{\mathbb CP^2}$.
\end{cor}

\begin{proof}
We can write $X=X_m\#n_{sm}\overline{\mathbb CP^2}$ by Lemma \ref{n=t}.  Hence by Lemma \ref{nge2} or \ref{n=1} it follows that  $M\cong X_m\#(n_{sm}-1)\overline{\mathbb CP^2}$.
\end{proof}

Note that this result is stronger than obtained in Lemma \ref{nge2}, as we have precise knowledge of the number of exceptional curves in $X$, whereas Lemma \ref{nge2} also applies to $\kappa(X)=-\infty$, for which no such precise statement exists.

\begin{lemma}\label{n4}
Assume that $n_{sm}> 4$ or $n_{sy}\ge 4$.  Then $\kappa(X)=\kappa(M)=-\infty$.
\end{lemma}

\begin{proof}
Assume $n_{sm}> 4$.  Then Lemma \ref{n=t} and the fact that $n_{sm}=4$ for artificial spheres in manifolds with $\kappa(X)\ge 0$ implies that $\kappa(X)=-\infty$.  Then it follows from Lemma \ref{nge2} that $\kappa(M)=-\infty$.

This leaves only the case $n_{sy}=n_{sm}=4$.  Hence, if $\kappa(X)\ge 0$, by Lemma \ref{n=t}, $V_X$ is artificial.  However, artificial spheres in positive Kodaira dimension have $n_{sy}=3$, thus again $\kappa(X)=-\infty$ and hence $\kappa(M)=-\infty$.

%Assume that $\kappa(X)\ge 0$.  Then $n_{sy}=n_{sm}>0$.  Thus Theorem \ref{minimal} shows that $M$ is not minimal.  More precisely, each pair of disjoint exceptional curves combines with the class $[H]$ in $\mathbb CP^2$ to produce a -1-curve in $M$.  Any two disjoint pairs in $X$ will produce two exceptional curves in $M$ which intersect once, due to $[H]^2=1$ and this intersection occurs away from  $2H$.  Hence we produce a manifold $M$ which has exceptional curves which intersect non-trivially.  Therefore $\kappa(M)=-\infty$ (see Lemma \ref{inters}).
%
%On the other hand, we can blow down the $n_{sm}$ curves in $X$.  If any of the $n_{sm}$ exceptional curves meet, then we already have $\kappa(X)=-\infty$.  Hence assume that they are all disjoint.  Then the -4-sphere contracts to an embedded sphere of square $n-4\ge 0$, hence $X$ also has $\kappa(X)=-\infty$ by Thm 1.4, \cite{M}.
\end{proof}

Moreover, if $n_{sy}=3$, then we obtain

\begin{lemma}\label{n3}Assume $X$ contains a symplectic $-4$-sphere and $\kappa(X)\ge 0$.  Let $n_{sy}=3$.  Then $V_X$ is artificial, $X=X_m\#4\overline{\mathbb CP^2}$ and $M\cong X_m\#3\overline{\mathbb CP^2}$.  
\end{lemma}

\begin{proof}
$X$ contains a symplectic $-4$-sphere which meets each of the $n_{sy}$ exceptional spheres in a single point.    After blowing down the 3 exceptional spheres we have reduced $V_X$ to an exceptional sphere itself while not changing any other exceptional spheres meeting $V_X$ in more than 2 points.  As $\kappa(X)\ge 0$ and hence no two distinct exceptional spheres intersect (Lemma \ref{inters}(2)), this implies there exist no exceptional curves $E$ with $[E]\cdot [V_X]\ge 2$.  Moreover, we can generically exclude the existence of any spheres of self-intersection $\le -2$.  Thus no new exceptional spheres can be produced in the blow-down with the exception of the single sphere produced from $V_X$.  Hence after blowing down the remnant of $V_X$ we obtain a minimal manifold $X_m$.
\end{proof}

\begin{remark}
Lemma \ref{n=t} proves $n_{sm}<4$ unless $V_X$ is artificial.  Lemma \ref{n3} together with Cor. \ref{n=n} then show, that the case $n_{sm}=3$ does not occur for a manifold with $\kappa(X)\ge 0$.

\end{remark}

\section{\label{rat}$-4$-blow-down of Rational Manifolds}

Cor. \ref{irrat} addressed the issue of $-4$-blow-downs in irrational ruled manifolds.  In the rational case, the behavior is more complex and provides interesting results.
 
Assume $X$ is a rational manifold.  Consider first a $-4$-blow-down producing a manifold $M$ with $\kappa(M)=-\infty$.  As $b_1$ is unchanged under $-4$-blow-downs, we find that $M$ must be rational again.  This is of course clear if $n_{sm}>0$.

However, a note of caution: The smooth results from Lemma \ref{flip} rely on the blowing down of an exceptional curve $E$ with  $[E]\in N_{sm}$ and this action affects all exceptional curves meeting the curve $E$.  If $X$ is not rational or ruled, then Lemma \ref{inters} ensures that all elements in either $N_{sm}$ or $N_{sy}$ are orthogonal while in the rational case this is no longer guaranteed.  Hence, even if $n_{sm}\ge 2$, the resulting manifold could be minimal.

If $n_{sm}=0$, then the $-4$-blow-down manifold $M$ will be minimal and thus $M=\mathbb CP^2$ or $S^2\times S^2$.  Let $X=X_m\#k\overline{\mathbb CP^2}$ be a manifold with $\kappa(X)=-\infty$.  Assume that $V_X\subset X$ is a symplectic $-4$-sphere and $(X,V_X)$ is relatively minimal.  Then 
\begin{equation}\label{kred}
K_M^2=K_{X_m}^2-k+1.
\end{equation}
Hence, if $M=\mathbb CP^2$, we have $K^2_{X_m}=8+k$, hence $k=0$ or $1$.  The manifold $S^2\times S^2$ ($k=0$) contains a $-4$-sphere, but this is a section of a sphere bundle, hence the $-4$-blow-down along it leaves the diffeomorphism type unchanged.  Thus $X$ cannot be $S^2\times S^2$.  Moreover, $X=\mathbb CP^2\#\overline{\mathbb CP^2}$ ($k=1)$ does not contain such a sphere.  Thus $M=\mathbb CP^2$ is not possible.

If $M=S^2\times S^2$, then $K^2_{X_m}=7+k$, which leads to the three diffeomorphic manifolds $(S^2\times S^2)\#\overline{\mathbb CP^2}\cong(S^2\tilde\times S^2)\#\overline{\mathbb CP^2}=\mathbb CP^2\#2\overline{\mathbb CP^2}$.  Table \ref{orbits} shows that $\mathbb CP^2\#2\overline{\mathbb CP^2}$ contains no symplectic $-4$-sphere with $n_{sm}=0$.

\begin{lemma}
If $\kappa(M)=-\infty$, then $n_{sm}>0$ and $M$ is diffeomorphic to the smooth blow-down of $X$.
\end{lemma}

Now consider $-4$-blow-downs with $\kappa(M)\ge 0$.  Clearly $n_{sm}=0$ must hold.  Given a manifold with $\kappa(X)=-\infty$, the results in Lemma \ref{rbd} strictly limit the possible candidates for $X$ which might produce a manifold $M$ with $\kappa(M)\ge 0$.  Thus we need $K_{X_m}^2\ge -1$ in order to obtain a new manifold $M$ with $\kappa(M)\ge 0$, limiting our choices to $\mathbb CP^2\#k\overline{\mathbb CP^2}$ with $k\le 10$. ($(S^2\times T^2)\#k\overline{\mathbb CP^2}$ with $k=0,1$ contains no symplectic $-4$-spheres and any symplectic $-4$-sphere in $S^2\times S^2$ is a section, hence the $-4$-blow-down has $\kappa(M)=-\infty$.)
 
Assume that $k\le 9$.  For such small blow-ups of $\mathbb CP^2$ we will show that either $n_{sy}\ge 1$ or the class $\xi$ with $\xi^2=-4$ and representable by a smoothly embedded sphere is not symplectically representable as an embedded sphere.  Thus we will have proven

\begin{lemma}\label{rat1}
The $-4$-blow-down $M$ of a rational manifold $X$ has $\kappa(M)\le 1$.
\end{lemma}

This result will rely in large part on the results of Wall \cite{W1} which state that for such small blow-ups, the automorphism group on homology induced by the action of diffeomorphisms is precisely the automorphism group of the homology lattice which preserves the intersection form.  In particular, the automorphism group is generated by the trivial automorphisms, namely interchanging two exceptional curves and $e_i\rightarrow -e_i$, as well as the reflection across $[H]-e_1-e_2$ ($k=2$) or $[H]-e_1-e_2-e_3$ ($3\le k\le 9$).  Using this result, a complete listing of all orbits of classes $\xi$ with $\xi^2=-4$ and representable by smoothly embedded spheres was given in \cite{bltjl}.  Table \ref{orbits} lists representatives for each orbit in simplified form $(a,b_1,...,b_n)$ which denotes the class $\xi=a[H]-\sum_{i=1}^kb_ie_i$ where $e_i$ are the generators of the homology of each $\overline{\mathbb CP^2}$ summand in the representation $X=\mathbb CP^2\#k\overline{\mathbb CP^2}$ and $a\ge 0$, $b_1\ge b_2\ge...\ge b_k$ and 
\begin{eqnarray}
\nonumber 2a\le b_1+b_2&k=2,\\
\nonumber 3a\le b_1+b_2+b_3&k\ge3,  
\end{eqnarray} 
see \cite{bltjl} for details.

As the intersection form is preserved and our interest is directed towards relatively minimal pairs, we may drop all $b_i=0$ and consider the hypersurfaces representing $\xi$ to lie in the manifold obtained by blowing down these $e_i$. 

Assume that $X$ carries a symplectic structure $\omega$ with canonical class $K$ making $V$ symplectic.  Prop. 3.9, \cite{LL3}, states, that diffeomorphisms act transitively on the set of symplectic canonical classes.  Hence there exists a diffeomorphism $\phi:(X,V,\omega,K)\rightarrow (\tilde X,\tilde V,\tilde\omega,K_{st})$ with $K_{st}=-3[H]+\sum \tilde e_i$.  The exceptional spheres $\tilde e_i$ are all symplectically representable in $\tilde X$ with respect to $\tilde \omega$, see Thm. A, \cite{LL3}.  In particular, the classes $e_i$ get mapped to $\tilde e_i$.  However, up to a permutation of the indices, $\tilde e_i=\pm e_i$.  More precisely, Theorem 1, \cite{TJL2} provides an orientation-preserving diffeomorphism $\phi_i:(X,V,\omega)\rightarrow (\tilde X,\tilde V,\tilde\omega)$ such that $e_i$ is $\mathbb Z$-homologous to a symplectic exceptional sphere, up to sign.  Prop 3.6, \cite{LL2} shows that using this procedure a diffeomorphism $\phi$ can be constructed such that the canonical class associated to $\omega$ gets mapped to the class $\tilde K=\pm3H\pm e_1\pm...\pm e_k$.  Using trivial automorphisms, Thm. 1, \cite{LL2}, shows that we can find an orientation preserving diffeomorphism mapping $K$ to $K_{st}$ as described above.

Under this diffeomorphism, $V$ is mapped to $\tilde V$ and the $-4$-blow-downs of $(X,V,\omega)$ and $(\tilde X,\tilde V,\tilde \omega)$ will be diffeomorphic, hence have the same Kodaira dimension.  Thus it suffices to understand the structures for the pair $(\tilde X,\tilde V)$.

Consider any class in Table \ref{orbits} which contains an entry $b_i=1$.  This indicates that $e_i\cdot \xi=1$.  However, assuming that $\xi$ is symplectically representable, this does not yet suffice to ensure $n_{sm}>0$ as we do not know that $e_i$ is symplectically representable nor do we know the intersection pattern of the smooth sphere representing $e_i$ and the hypersurface representing $\xi$.  Were this the case, i.e. $\tilde e_i=e_i$ corresponding to $b_i=1\mapsto \tilde b_i=1$, we would have $n_{sy}\ge 1$ and thus, under the assumption that $\xi$ is symplectically representable,  $\kappa(M)=\kappa(X)$.  

In any example in Table \ref{orbits} with $b_i=1$, we shall therefore assume from now on that this is not the case, i.e. $b_i=1\mapsto \tilde b_i=-1$.  This means, using the above diffeomorphism, we obtain $\tilde e_i=-e_i$ is symplectically representable and thus $\tilde b_i=\tilde e_i\cdot[\tilde V]=-e_i\cdot [\tilde V]=-1$.  Hence we may write $\tilde\xi$ as $(\pm a,\pm b_1,...,\pm b_j, -1,..,-1)$.  From this we can calculate all possible values of $ K_{st}\cdot\tilde\xi$, they lie in the set of numbers $\{\pm 3a\pm b_1\pm...\pm b_j-1-....-1\}$.  If $\xi$ is to be symplectically representable, then this set must contain $2$, as follows from the adjunction equality.

\begin{example}\label{large}Consider the class $\xi=(5,4,2,2,2,1)$ in Table \ref{orbits}.  The possible values of $ K_{st}\cdot \tilde\xi$ under the above assumptions are $\{\pm 15\pm4\pm2\pm2\pm2-1\}$, which does not contain $2$.  Hence this class is not symplectically representable.

Two basic situations can appear:  The first term is too large in comparison to the rest of the terms.  For example, the leading term in $(9,8,2,2,2,2,2,1)$ leads to $\pm 27$, which cannot be adequately reduced by the remaining terms. 

In the second case, the possible values of $ K_{st}\cdot \tilde\xi$ contains 0 and all further positive values occur in multiples of $4$.  This is due to the occurrence of only $-1$ and $\pm 2$ as values for $b_i$.  For example, the possible values of $ K_{st}\cdot \tilde\xi$ for the class $(2,2,1,1,1,1)$  are $\{-12,-8,0,4\}$ or for the class $(3,2,2,1,1,1,1,1)$ the values $\{-18,-14,-10,0,4,8\}$.

This same phenomenon can be observed for a number of classes in Table \ref{orbits}.
\end{example}

\begin{example}\label{9}
The lone example for $k=9$, namely $(5,3,2,2,2,2,1,1,1,1)$ is not symplectically representable.  Calculating the possible values for $ K_{st}\cdot \tilde\xi$ leads to the values $\{...-10,0,4,6,10...\}$ around the needed value $2$. 

\end{example}

In other cases we obtain multiples of an exceptional sphere;
\begin{example}\label{mult}Consider the class $(4,2,2,2,2,2)$.  The value $ K_{st}\cdot \tilde\xi=2$ is obtained for $\tilde \xi=-4[H]+2\sum_{i=1}^5e_i=-2(2[H]-\sum_{i=1}^5e_i)$.  Hence $\tilde \xi$ is a negative multiple of a symplectic exceptional sphere and itself cannot be symplectically representable.

The same holds for $(6,4,2,2,2,2,2,2)$.

\end{example}

We are not interested in any classes with $n_{sm}>0$.  Assuming that the classes $\xi$ are symplectically representable, we can use the above constructions to pick out a configuration with $K_{st}\cdot \tilde \xi =2$ and consider the value of $n_{sy}$.

\begin{example}\label{nsy}
Consider the class $\xi=(3,2,2,2,1)$.  The possible set of values for $K_{st}\cdot \tilde \xi$ is given by $\{\pm9\pm2\pm2\pm2-1\}$ which contains $2$ given by $9-2-2-2-1$.  This corresponds to the class $\tilde \xi=-3[H]+2\tilde e_1+2\tilde e_2+2\tilde e_3+\tilde e_4$.  With $\tilde e_i$ symplectically representable we also obtain the symplectic exceptional spheres $\tilde e_{ij}=[H]-\tilde e_i-\tilde e_j$.  However, $\tilde e_{12}\cdot \tilde \xi=1$, hence $n_{sm}\ge n_{sy}\ge 1$.  Thus the class $\xi$, assuming it is symplectically representable, will lead to $\kappa(M)=\kappa(X)=-\infty$.

Again, this holds true for a number of examples in Table \ref{orbits}, keeping in mind that as $k$ increases, we obtain exceptional spheres of the form $e_i$, $[H]-e_i-e_j$, $2[H]-\sum_5e_i$, $3[H]-2e_i-\sum_6e_j$ and more complicated sums, see \cite{bltjl2}.

\end{example}

In the previous examples we do not know which classes are symplectically representable, even if we drop the condition on the exceptional classes $e_i$. 

\begin{example}\label{sym}
For some of the classes the diffeomorphism leads to classes which we know to be symplectically representable, however only with $b_i=1\mapsto \tilde b_i=1$.  For example, the class $(1,2,1)$ can be mapped to the class $(-1,-2,1)$.  There exists a symplectic $\tilde \omega$ such that this class is symplectically representable:  $[\tilde V]=(-1,-2,1)$ is represented by the blow-up of an embedded symplectic $-3$-sphere in $\mathbb CP^2\#\overline{\mathbb CP^2}$.

Similar results hold for the classes $(0,1,1,1,1)$ (This is an artificial $-4$-sphere.), $(1,1,1,1,1,1)$ and $(2,1,1,1,1,1,1,1,1)$, see the examples in \cite{D}.

Note that for all of these examples we can find a symplectic exceptional sphere meeting $V$ in a single positive transverse point.  Hence $n_{sm}>0$ in all these examples and thus $\kappa(M)=\kappa(X)=-\infty$.
\end{example}

This completes the cases found in Table \ref{orbits} and proves Lemma \ref{rat1}. 

 We are left with embedded symplectic $-4$-spheres in $X=\mathbb CP^2\#10\overline{\mathbb CP^2}$.  In the examples found in \cite{D}, we have seen that the classes $-K$ and $-2K$ in $\mathbb CP^2\#9\overline{\mathbb CP^2}$ can be blown up in a single double point to obtain an embedded $-4$-sphere in $\mathbb CP^2\#10\overline{\mathbb CP^2}$.  These two cases will be discussed in the next section, the first leading to a logarithm two transform of $E(1)$ and hence to no change in Kodaira dimension, the second to the Enriques surface with $\kappa(M)=0$.

The possible existence of other classes $\xi$ with $\xi^2=-4$ in $\mathbb CP^2\#10\overline{\mathbb CP^2}$ admitting symplectic spherical representatives cannot be ruled out. However, we conjecture that no such curves exist, see Conjecture \ref{conj}.

\subsection{$\bf \kappa(X)=-\infty$ $\bf \Rightarrow$ $\bf \kappa(M)=0$\label{kod0}}

Which configurations produce a manifold $M$ with $\kappa(M)=0$?  Clearly, we must have $X=\mathbb CP^2\#10\overline{\mathbb CP^2}$.  The $-4$-blow-down of a symplectic pair $(X,V_X)$ would produce an Enriques surface, as $b_1$ and $b^+$ are unchanged.  A suitable $V_X$ must now be found.

To this end, we consider the following result, which follows from the proof of Thm. 3.1, \cite{U2}:

\begin{lemma}\label{kv}
Assume that $M=X\#_{V_X=2H}\mathbb CP^2$ is minimal and has Kodaira dimension 0.  Then $K_X=-\frac{1}{2}[V_X]$.
\end{lemma}

\begin{proof}
The proof of Thm. 3.1, \cite{U2}, shows that $K_X=\mu_1[V_X]$ and $K_{\mathbb CP^2}=\mu_22[H]$ if $M$ is minimal with Kodaira dimension 0.  Moreover, 
\[
(2+\mu_1+\mu_2)\int_{V_X}\omega_X=0
\]
which implies $\mu_1+\mu_2=-2$.  As $\mu_2=-\frac{3}{2}$, it follows that $\mu_1=-\frac{1}{2}$.
 
\end{proof}

This implies that $n_{sm}=0$ as $[V_X]\cdot E=-2K_X\cdot E\ne 1$ for any smooth exceptional classes $E$.  Moreover, given any symplectic exceptional class $E$ we obtain that $[V_X]\cdot E=2$.

We now determine the class of $V_X$ explicitly:  Let $X=\mathbb CP^2\#10\;\overline{\mathbb CP^2}$ be endowed with a symplectic form $\omega$.  Associate to $\omega$ the symplectic canonical class $K_\omega$.  Theorem 1, \cite{LL2}, provides an orientation-preserving diffeomorphism $\phi:X\rightarrow X$ such that $\phi^*K_\omega=K_{st}=-3[H]+\sum e_i$.  Moreover, Theorem D, \cite{LL3}, states that the symplectic form $\phi^*\omega$ is equivalent to $\omega_{st}$ which is obtained from the symplectic form on $\mathbb CP^2$ through symplectic blow-ups.  Equivalent means up to deformation and diffeomorphisms.  Hence, given the standard symplectic structure on $X$ with canonical class $K_{st}=-3[H]+\sum e_i$ we obtain $[V_X]=6[H]-2\sum_{i=1}^{10}e_i$.  This is precisely one of the examples noted in \cite{D}.

Recall the fact that this class is symplectically representable.  This can be seen as follows:  The class $6[H]$ in $\mathbb CP^2$ admits a non-vanishing Gromov-Witten invariant in genus $g=0$, a curve representing this class is an immersed sphere with 10 double points.  Blowing up each of these double points provides for the class $-2e_i$ while removing the double point.  The final result is an embedded sphere in class $6[H]-2\sum_{i=1}^{10}e_i$, which is then the symplectic $-4$-sphere along which we blow-down.  Thus we have proven the following Lemma:

\begin{lemma}
Let $X=X_m\#k\overline{\mathbb CP^2}$ be a manifold with $\kappa(X)=-\infty$.  Assume that $V_X\subset X$ is a symplectic $-4$-sphere and $(X,V_X)$ is relatively minimal.  Then the $-4$-blow-down $M$ along $V_X$ has $\kappa(M)\ge 0$ only if $X=\mathbb CP^2\#10\overline{\mathbb CP^2}$.  Moreover, if $[V_X]=6[H]-2\sum_{i=1}^{10}e_i$, we can obtain a symplectic manifold $M$ with  $\kappa(M)=0$.

\end{lemma}

We now consider the latter case in detail.

\subsection{The $-4$-blow-down of the pair $X=\mathbb CP^2\#10\overline{\mathbb CP^2}$ and $[V_X]=6[H]-2\sum_{i=1}^{10}e_i$}

The goal of this section is to show that the pairs $(X,V_X)=(\mathbb CP^2\#10\overline{\mathbb CP^2}, V_X)$ with $[V_X]=6[H]-2\sum_{i=1}^{10}e_i$ and $(E(1)_2\#\overline{\mathbb CP^2},V_2)$ with $[V_2]=f-2e_b$ are symplectomorphic (here $E(1)_2$ denotes a logarithmic transform of order 2 of $E(1)$).  It then follows, that the $-4$-blow-down of $(X,V_X)$ has the diffeomorphism type of the Enriques surface.

Note that an Enriques surface is a two-fold logarithmic transform of $E(1)$ of order 2, see \cite{GS}.  It can be produced from $E(1)$ with two $-4$-blow-downs as follows:
\[
\begin{diagram}
\node{E(1)}\arrow{e,t}{blow-up}\node{E(1)\#\overline{\mathbb CP^2}}\arrow{e,tb}{-4}{ blow-down}\node{E(1)_2}\arrow{e,t}{blow-up}\node{E(1)_2\#\overline{\mathbb CP^2}}\arrow{s,lr}{-4}{ blow-down}\\
\node[4]{E(1)_{2,2}}
\end{diagram}
\]

 The logarithmic transformation $E(1)_2$ is just the $-4$-blow-down of the $-4$-sphere $F-2E_b$ in $E(1)\#\overline{\mathbb CP^2}$ (Note that $F=-K$ in $E(1)$, recall the discussion in the previous section.).  To produce an Enriques surface, we now need to blow up a cusp-fiber $f$ in $E(1)_2$ at the nodal point to produce an embedded $-4$-sphere $V_2$ in the class $f-2e_b$ in $E(1)_2\#\overline{\mathbb CP^2}$, noting that $n_{sy}=0$ for any $-4$-sphere in this class.  This we rationally blow-down to produce the Enriques surface $E(1)_{2,2}$.
 
 Due to  Lemma \ref{kv}, we know that the class $[V_2]$ of this $-4$-sphere  is $-2K_{E(1)_2\#\overline{\mathbb CP^2}}$.  This can be seen directly:  The class $f-2e_b$ can be written as $2(f_2-e_b)$ which is $-2K_{E(1)_2\#\overline{\mathbb CP^2}}$, where $f_2$ denotes the class of the new multiple fiber.

Clearly, the underlying smooth manifolds $X$ and $E(1)_2\#\overline{\mathbb CP^2}$ are diffeomorphic:  Any logarithmic transform of $E(1)$ is diffeomorphic to $E(1)$ (Thm 8.3.11, \cite{GS}) and hence to $\mathbb CP^2\#9\overline{\mathbb CP^2}$.

Using these diffeomorphisms (after symplectically blowing up a point on the respective cusps) we can pull-back $\omega_{st}$ from $X$ to $E(1)_2\#\overline{\mathbb CP^2}$ to obtain a symplectic form $\tilde\omega$.  Now apply Theorem 1, \cite{LL2}, and Theorem D, \cite{LL3}, as before to obtain the necessary diffeomorphism taking $\tilde\omega$ to a symplectic form $\omega$ with canonical class $K_{E(1)_2\#\overline{\mathbb CP^2}}=-\frac{1}{2}[V_2]=f_2-e_b$.  This constructs a symplectomorphism of pairs from $(X,V_X,\omega_{st})$ to   $(E(1)_2\#\overline{\mathbb CP^2},V_2,\omega)$

%Therefore, we need to relate the respective manifolds along which the $-4$-blow-down is performed.  To this end, we determine precisely how the $-4$-sphere in the Enriques case is produced and the intersection pattern of the exceptional curves with this sphere.
%
%  Our understanding of the $-4$-blow-down of a $-4$-sphere shows, that any two exceptional curves in $E(1)$ which meet the fiber $F$ in a single positive point will combine to become a new exceptional curve in $E(1)_2$.  However, the fiber class $F$ of $E(1)$ is just $-K$, hence every exceptional curve in $E(1)$ meets $F$ in a single positive point.  Thus all exceptional curves, except possibly those generated in the blow up of $E(1)$, meet the fibers $f$ (the new multiple fiber is denoted $f_2$) of $E(1)_2$ in 2 positive points.  
%
%  
%
% In order for an exceptional curve $e$ in $E(1)_2$ to meet $f$ in a single positive point, it must be generated from an exceptional curve $E$ in $E(1)$ and a new exceptional curve $E'$ generated in the blow-up of $E(1)$, but with $E'\cdot F=0$.  But any such exceptional curve $E'$ does not meet the $-4$-sphere $F-2E_b$ in a single positive point.  Hence we obtain no exceptional curves meeting $f$ in a single point.  

The diffeomorphism between $E(1)_2\#\overline{\mathbb CP^2}$ and $X$ ensures that the $-4$-blow-downs of the respective $-4$-spheres will produce manifolds of the same diffeomorphism type.

\section{$-4$-blow-downs from $X$ with $\kappa(X)\ge 0$}

The results of Section \ref{kodsection} allow us to analyse the structure of $M$.  We assume that $V_X$ is not artificial.  Due to Cor. \ref{n=n}, we write $n=n_{sm}=n_{sy}$ in the following.

\begin{lemma}\label{k=0}Assume $\kappa(X)=0$.  Then the following holds:
\begin{enumerate}
\item If $n=0$, then $X=X_m\#\overline{\mathbb CP^2}$ and $\kappa(M)=1$.  Furthermore, $M$ is an order 2 logarithmic transform of $X_m$.
\item If $n>0$ then $n=2$.  Moreover, $\kappa(M)=0$ and $M$ is diffeomorphic to  $X_m\#\overline{\mathbb CP^2}$ with $X_m$  either a K3 surface, an Enriques surface or an unknown surface with $\kappa(X)=0$.  There exist no non-artificial examples with $n\ne 2$.
\end{enumerate}
In particular, $\kappa(M)\le 1=\kappa(X)+1$.
\end{lemma}

\begin{proof}Assume $n=0$.  By Lemma \ref{k=t}, we must have $k=0$ or $1$.  No minimal manifold with Kodaira dimension 0 contains a symplectic $-4$-sphere.    Thus we have no examples with $k=0$.   If $k=1$, then the adjunction formula implies that $E\cdot [V_X]=2$, hence $V_X$ is obtained from the blow-up of a cusp fiber at a nodal point.  Lemma \ref{Kw} shows that $K\cdot\omega$ stays positive, while the blow-up and subsequent $-4$-blow-down leaves $K^2$ unchanged.  Hence  $\kappa(M)=1$.  Ex. 1 in Section 3 of \cite{FS} describes the structure of $M$ as a logarithmic transform of $X$.

Assume that $n\ge 1$.  $n=1$ is excluded by the adjunction formula  and $n\ge 4$ by Lemma \ref{n4}.  This leaves $n=2$, which, by Lemma \ref{n=t}, leads to $n=k=2$.  In particular, the minimal manifold $X_m$ must contain an embedded $-2$-sphere.  Of the known manifolds with $\kappa(X)=0$, only the K3 surface and the Enriques surface admit such hypersurfaces.  Moreover,  $M$ is diffeomorphic to  $X_m\#\overline{\mathbb CP^2}$ by Cor. \ref{n2blow}.   
\end{proof}

To all of the cases above there exist examples:  If $n=0$, then the blow up of a cusp fiber ensures that examples exist of this type.  For $n=2$, see \cite{D}.

%\begin{example}
%Recall Ex. \ref{K3}.  In this example we took a manifold with $\kappa(X)=0$ and performed the rational blow down.  We started with a Kummer surface, a K3 surface with an embedded sphere $S$ with square $-2$, and blew up two points on this sphere to obtain a $-4$-sphere.  Assume that we have chosen an elliptic Kummer surface such that $S$ is a section of the fibration.  Each fiber intersects the section once and continues to do so after blowing up.  The rational blow down $M$ then contains an embedded symplectic genus 2 surface with self-intersection 1, produced by combining the two fibers with a sphere $H$ in $\mathbb CP^2$. This surface intersects the exceptional divisor in a single point, the blow-down of this surface is a genus 2 surface with self-intersection 2.  This surface is an embedded symplectic surface
%
%Comment on the genericity of J
%
%
%\end{example}

In higher Kodaira dimensions we cannot make such explicit statements on the structure of $M$, however we can explicitly determine the Kodaira dimension based on $n$ and $k$:

\begin{lemma}
Assume $n=0$ and $\kappa(X)=1$.  Then the following holds:
\begin{enumerate}
\item If $k=0$, i.e. $X$ is minimal, then $\kappa(M)=2$.
\item If $k=1$, then $\kappa(M)=1$.
\end{enumerate}

\begin{proof}
The result follows from Lemma \ref{rbd} as well as Thm. \ref{main}:
\[
K_{M_m}^2=K_X^2+1=\left\{
\begin{array}{cc}
K^2_{X_m}+1&k=0\\K_{X_m}^2&k=1
\end{array}\right.\:=\:\left\{
\begin{array}{cc}
1&k=0\\0&k=1
\end{array}\right.
\]

\end{proof}

\end{lemma}

Again, examples can be found with these properties:  For $k=0$, the elliptic surface $E(4)$ has $\kappa=1$ and $n=k=0$ while containing a symplectic $-4$-sphere.  This example is described in \cite{G} and the $-4$-blow-down is shown to have $K^2=1$.  The $k=1$ case can again be obtained from the blow-up of a cusp fiber in any of the elliptic surfaces $E(n)$ for $n\ge 3$.  Note that this is again a logarithmic transform of order 2 and the resulting manifold $M=E(n)_2$ is not diffeomorphic to $E(n)\#t\overline{\mathbb CP^2}$ for any $t\ge 0$ (Thm 8.3.12, \cite{GS}) even though there is no change in the Kodaira dimension.

\newpage
\begin{figure}
\begin{math}
\begin{array}{|c|l|c|c|c|}\hline
\mbox{rel. min }k&\xi=(a,b_1,...,b_k)&\mbox{sympl. rep.}&n_{sm}&\\\hline
1&(0,2)&N& &\mbox{non-trivial }S^2\mbox{-bundle}\\\hline
2& (1,2,1)^*&N& &\mbox{Ex. } \ref{large}\\\hline

%3&(1,2,1,0)&0&y& >0\\\hline
%&(0,2,0,0)&2&n& \\\hline
4&(0,1,1,1,1)^*&N&  &\mbox{Ex. } \ref{large}\\\hline
&(3,2,2,2,1)&&>0 &\mbox{Ex. } \ref{nsy}\\\hline
5&(1,1,1,1,1,1)^*&N& &\mbox{Ex. } \ref{large}\\\hline
&(2,2,1,1,1,1)&N&&\mbox{Ex. } \ref{large}\\\hline
&(4,2,2,2,2,2)&&  >0&\mbox{Ex. } \ref{mult}\\\hline
&(5,4,2,2,2,1)&N& &\mbox{Ex. } \ref{large}\\\hline
6&(4,2,2,2,2,1,1)&& >0&\mbox{Ex. } \ref{nsy}\\\hline
&(4,3,2,2,1,1,1)&& >0&\mbox{Ex. } \ref{nsy}\\\hline
&(5,3,2,2,2,2,2)&&>0&\mbox{Ex. } \ref{nsy} \\\hline
&(7,6,2,2,2,2,1)&N&&\mbox{Ex. } \ref{large}\\\hline
7&(3,2,2,1,1,1,1,1)&N&&\mbox{Ex. } \ref{large}\\\hline
%&(3,2,2,1,1,1,1,1)&0&&>0 \\\hline
%&(3,2,2,1,1,1,1,1)&0&& >0\\\hline
&(6,4,2,2,2,2,2,2)&& >0&\mbox{Ex. } \ref{mult}\\\hline
&(6,5,2,2,2,1,1,1)&N&&\mbox{Ex. } \ref{large}\\\hline
&(9,8,2,2,2,2,2,1)&N&&\mbox{Ex. } \ref{large}\\\hline
8&(2,1,1,1,1,1,1,1,1)^*&N& &\mbox{Ex. } \ref{large}\\\hline
&(4,2,2,2,2,1,1,1,1)&N&&\mbox{Ex. } \ref{large}\\\hline
&(5,2,2,2,2,2,2,2,1)&N&&\mbox{Ex. } \ref{large}\\\hline
&(5,4,2,2,1,1,1,1,1)&& >0&\mbox{Ex. } \ref{nsy}\\\hline

&(7,5,2,2,2,2,2,2,2)&& >0&\mbox{Ex. } \ref{nsy}\\\hline
&(8,7,2,2,2,2,1,1,1)&N&&\mbox{Ex. } \ref{large}\\\hline
%&(4,2,2,2,2,1,1,1,1)&-2&& >0\\\hline
%&(4,2,2,2,2,1,1,1,1)&-2&& >0\\\hline
%&(4,2,2,2,2,1,1,1,1)&-2&&>0 \\\hline

%&(5,2,2,2,2,2,2,2,1)&0&& >0\\\hline
%&(5,2,2,2,2,2,2,2,1)&0&&>0 \\\hline
&(11,10,2,2,2,2,2,2,1)&N&&\mbox{Ex. } \ref{large}\\\hline
9&(5,3,2,2,2,2,1,1,1,1)&N&&\mbox{Ex. } \ref{9}\\\hline

%&(1,2,1,0,0,0,0,0,0,0)&0&y& >0\\\hline
%&(0,2,0,0,0,0,0,0,0,0)&2&n& \\\hline

\end{array}
\end{math}
\caption{\label{orbits}Orbits of $-4$-classes under the action of the group of automorphisms of $H_2$ for small blow-ups ($k\le 9$) for rational manifolds.  The class $\xi$ is a simplified representative for the orbit.  Under the assumption $b_i=1\mapsto \tilde b_i=-1$, the results from this table fall into two categories: Hypersurfaces with $n_{sm}>0$ in the fourth column, which may or may not be symplectically representable, and hypersurfaces which are not symplectically representable ($N$ in the third column).  Some classes are known to be symplectically representable ($*$) but have $n_{sm}>0$ (under the asumption $b_i=1\mapsto \tilde b_i=1$, see Ex. \ref{sym}, Section \ref{rat}).   }

\end{figure}

\end{document}